\documentclass[a4paper,11pt,reqno]{amsart}
\usepackage{amsmath}%
\usepackage{amsthm}%
\usepackage{amsfonts}%
\usepackage{amssymb}%
\usepackage[english]{babel}
\usepackage{graphicx}%
\usepackage{easymat}%
\usepackage{bbm}
\usepackage{inputenc}
\usepackage{a4wide}
\usepackage{cite}
\usepackage{epsfig}
\usepackage{color}

\usepackage{etex}

\usepackage{todonotes}
\usepackage{hyperref}
\usepackage[all]{xy}
\newtheorem{theorem}{Theorem}[section]

\newtheorem{assumption}[theorem]{Assumption}

\newtheorem{definition}[theorem]{Definition}

\newtheorem{lemma}[theorem]{Lemma}

\newtheorem{remark}[theorem]{Remark}

\numberwithin{equation}{section}

\def\beq{\begin{equation}}
\def\ee{\end{equation}}

\def\CC{\mathbb{C}}

\def\RR{\mathbb{R}}

\def\ZZ{\mathbb{Z}}

\def\unity{\mathbbm{1}}
\def\del{\partial}

\def\trace{\text{tr}}

\begin{document}
\title{Darboux coordinates for periodic solutions of the sinh-Gordon equation}

\date{\today}

\author{Markus Knopf}

\address{Institut f\"ur Mathematik, Universit\"at Mannheim, 
68131 Mannheim, Germany.} 

\email{knopf@math.uni-mannheim.de}

\begin{abstract}
\footnotesize
We study the space of periodic solutions of the elliptic $\sinh$-Gordon equation by means of spectral data consisting of a Riemann surface $Y$ and a divisor $D$ and prove the existence of certain Darboux coordinates.
\end{abstract}

\maketitle


\section{Introduction}

The elliptic \textit{sinh-Gordon equation} is given by
\beq
\Delta u + 2\sinh(2u) = 0,
\label{eq_sinh}
\ee
where $\Delta$ is the Laplacian of $\RR^2$ with respect to the Euclidean metric and $u:\RR^2 \to \CC$ is a twice partially differentiable complex-valued function. \\

In the present setting we only demand that $u$ is periodic with one fixed period. After rotating the domain of definition we can assume that this period is real. This enables us to introduce the space $M^{\mathbf{p}}$ of simply periodic \textit{Cauchy data} with fixed period $\mathbf{p} \in \RR$ consisting of pairs $(u,u_y) \in W^{1,2}(\RR/\mathbf{p}\ZZ) \times L^2(\RR/\mathbf{p}\ZZ)$. \\

In \cite{Knopf_phd} the map $\Phi: (u,u_y) \mapsto (Y(u,u_y), D(u,u_y))$ was studied for finite type solutions $u$ of the $\sinh$-Gordon equation. $\Phi$ assigns spectral data consisting of a Riemann surface $Y(u,u_y)$ and a divisor $D(u,u_y)$ to the Cauchy data $(u,u_y)$ of such solutions. \\

We will restrict to the map $(u,u_y) \mapsto D(u,u_y)$ that assigns to Cauchy data $(u,u_y) \in M^{\mathbf{p}}$ a divisor $D(u,u_y) = \sum_i (\lambda_i,\mu_i)$ on the spectral curve $Y(u,u_y)$ to potentials $(u,u_y)$ where $D$ has only simple points, i.e. $D$ contains no points of higher order. Moreover, we will consider $\lambda_i,\, \mu_i$ as maps $\lambda_i,\, \mu_i: M^{\mathbf{p}} \to \CC$. \\

The main goal of this paper is to prove the existence of certain Darboux coordinates for $M^{\mathbf{p}}$ and to adapt Theorem 2.8 in \cite{Poeschel_Trubowitz}. More precisely, we show that $(\ln\lambda_i,\ln\mu_i)$ are indeed Darboux coordinates with respect to the symplectic form $\Omega: T_{(u,u_y)}M^{\mathbf{p}} \to \CC, \; ((\delta u, \delta u_y),(\widetilde{\delta}u,\widetilde{\delta}u_y)) \mapsto \Omega((\delta u, \delta u_y),(\widetilde{\delta}u,\widetilde{\delta}u_y))$ on the tangent space $T_{(u,u_y)}M^{\mathbf{p}}$, that was introduced in \cite{McKean}, \cite{Knopf_phd}. \\

P\"oschel and Trubowitz describe Theorem 2.8 in \cite{Poeschel_Trubowitz} as one of the main ingredients for the investigation of the KdV equation by means of spectral theory. Since we are able to adapt this theorem to our situation, we expect that broad parts of \cite{Poeschel_Trubowitz} can be carried over for the $\sinh$-Gordon equation, as well.

\section{Cauchy data and the monodromy} 
Let us consider the system
$$
\tfrac{\del}{\del x}F_{\lambda} = F_{\lambda}U_{\lambda}, \;\; \tfrac{\del}{\del y} F_{\lambda} = F_{\lambda} V_{\lambda} \;\; \text{ with } \;\;  F_{\lambda}(0) = \unity
$$
and
\beq
U_{\lambda} = \frac{1}{2}
\begin{pmatrix}
 -i u_y & i\lambda^{-1}e^u + ie^{-u} \\
i\lambda e^u + ie^{-u} & iu_y
\end{pmatrix}, \;\;\;
V_{\lambda} = \frac{1}{2} \begin{pmatrix} i u_x & -\lambda^{-1}e^u + e^{-u} \\ \lambda e^u - e^{-u} & -i u_x \end{pmatrix}.
\label{eq_U_lambda}
\ee
The compatibility condition for this system
$$
\tfrac{\del}{\del y} U_{\lambda} - \tfrac{\del}{\del x}V_{\lambda} - [U_{\lambda}, V_{\lambda}]= 0
$$
holds if and only if the function $u:\CC \to \CC$ satisfies the \textbf{sinh-Gordon equation}
\beq
\Delta u + 2\sinh(2u) = 0.
\label{eq_sinh}
\ee
$F_{\lambda}$ is called \textbf{extended frame} for the pair $(U_{\lambda}, \, V_{\lambda})$. 

\subsection{Cauchy data $(u,u_y)$}

We consider simply periodic solutions of \eqref{eq_sinh} with a fixed period $\mathbf{p} \in \CC$. After rotating the domain of definition we can assume that this period is real, i.e. $\Im(\mathbf{p}) = 0$. From now on we therefore consider simply periodic \textit{Cauchy data} with fixed period $\mathbf{p} \in \RR$ consisting of a pair $(u,u_y) \in W^{1,2}(\RR/\mathbf{p}\ZZ) \times L^2(\RR/\mathbf{p}\ZZ)$.

\begin{remark}\label{isomorphism_U_lambda}
Due to \eqref{eq_U_lambda} the matrix $U_{\lambda}$ uniquely determines the tuple $(u,u_y)$. Vice versa, the tuple $(u,u_y)$ determines $U_{\lambda}$ and $V_{\lambda}$.
\end{remark}

\subsection{The monodromy}

The central object for the following considerations is contained in

\begin{definition}
Let $F_{\lambda}$ be an extended frame for $U_{\lambda}$ and assume that $U_{\lambda} = F_{\lambda}^{-1}\frac{d}{dx}F_{\lambda}$ has period $\mathbf{p}$, i.e. $U_{\lambda}(x+\mathbf{p}) = U_{\lambda}(x)$. Then the \textbf{monodromy} of the frame $F_{\lambda}$ with respect to the period $\mathbf{p}$ is given by
$$
M_{\lambda}^{\mathbf{p}} := F_{\lambda}(x+\mathbf{p})F_{\lambda}^{-1}(x).
$$
\end{definition}

Since $F_{\lambda}(0) = \unity$, we get 
$$
M_{\lambda} := M_{\lambda}^{\mathbf{p}} = F_{\lambda}(\mathbf{p})F_{\lambda}^{-1}(0) = F_{\lambda}(\mathbf{p}).
$$

\section{Spectral curve $Y$ and divisor $D$}

The eigenvalues and eigenlines of the monodromy $M_{\lambda}$ are encoded in the so-called \textit{spectral data} $(Y(u,u_y), D(u,u_y))$. We omit the dependency on $(u,u_y)$ in the following. Let us first define the \textbf{spectral curve} $Y$ by the eigenvalues of $M_{\lambda}$:
$$
Y := \{(\lambda, \mu) \in \CC^* \times \CC \left|\right. \det(M_{\lambda} - \mu\unity) = 0\}
$$
$Y$ is a non-compact hyperelliptic Riemann surface with possible singularities and is equipped with the so-called hyperelliptic involution $\sigma: Y \to Y, \, (\lambda,\mu) \mapsto (\lambda, 1/\mu)$. \\

The eigenlines of the monodromy $M_{\lambda}$ are described by normalized eigenvectors.

\begin{lemma}\label{lemma_eigenbundle}
On the spectral curve $Y$ there exist unique meromorphic maps $v(\lambda,\mu)$ and $w(\lambda,\mu)$ from $Y$ to $\CC^2$ such that
\begin{enumerate}
\item[(i)] For all $(\lambda,\mu) \in Y$ the value of $v(\lambda,\mu)$ is an eigenvector of $M_{\lambda}$ with eigenvalue $\mu$ and  $w(\lambda,\mu)$ is an eigenvector of $M_{\lambda}^t$ with eigenvalue $\mu$, i.e. 
$$
M_{\lambda}v(\lambda,\mu) = \mu v(\lambda,\mu), \;\;\; M^t_{\lambda}w(\lambda,\mu) = \mu w(\lambda,\mu).
$$
\item[(ii)] The first components of $v(\lambda,\mu)$ and $w(\lambda,\mu)$ are equal to $1$, i.e. $v(\lambda,\mu) = (1, v_2(\lambda,\mu))^t$ and $w(\lambda,\mu) = (1, w_2(\lambda,\mu))^t$ on $Y$.
\end{enumerate}
\end{lemma}

Set $M_{\lambda} = \left(\begin{smallmatrix} a(\lambda) & b(\lambda) \\ c(\lambda) & d(\lambda) \end{smallmatrix}\right)$ with holomorphic functions $a, b, c, d : \CC^* \to \CC$. Then we get
$$
M_{\lambda} v(\lambda,\mu) = \mu v(\lambda,\mu) \;\; \text{ for } \;\; v(\lambda,\mu) = \begin{pmatrix} 1 \\ \frac{\mu-a}{b} \end{pmatrix} = \begin{pmatrix} 1 \\ \frac{c}{\mu - d} \end{pmatrix},
$$
where the last equations holds due to the equation $\det(M_{\lambda} - \mu\unity) = 0$ defining $Y$. Then the pole divisor of the function $ f:= \frac{\mu-a}{b} = \frac{c}{\mu - d}$ gives rise to the \textbf{divisor}
$$
D := \{(\lambda, \mu) \in \CC^* \times \CC \left|\right. b(\lambda) = 0 \text{ and } d(\lambda) = \mu \} \subset Y.
$$
In the finite type situation it is known that the spectral data $(Y,D)$ determine the monodromy $M_{\lambda}$ and also the corresponding Cauchy data $(u,u_y)$ (cf. \cite{Hitchin_Harmonic}). Due to the work of \cite{Klein} this is also the case for general Cauchy data $(u,u_y)$, possibly of infinite type. In fact, the divisor $D$ alone uniquely determines $(u,u_y)$ \cite{Klein}. We will need the following assumption.

\begin{assumption}\label{assumption_divisor}
All points $(\lambda_i,\mu_i)$ from $D$ are simple points, i.e. $D$ contains no points of higher order.
\end{assumption}

By definition of $D$, this assumption is equivalent to $b(\lambda)$ having pairwise distinct roots. \\

Dropping this assumption, one would have to work with elementary symmetric functions in order to obtain similar results as in our present setting.

\section{Hamiltonian formalism}

Simply periodic Cauchy data $(u,u_y) \in W^{1,2}(\RR/\mathbf{p}\ZZ) \times L^2(\RR/\mathbf{p}\ZZ)$ can be considered as a \textbf{symplectic manifold} $M^{\mathbf{p}}$ with a symplectic form $\Omega$ (following the exposition of \cite{McKean}). For $(u,u_y) \in M^{\mathbf{p}}$ and $(\delta u, \delta u_y),(\widetilde{\delta} u,\widetilde{\delta} u_y) \in T_{(u,u_y)}M^{\mathbf{p}}$ the symplectic form is given by
$$
\Omega((\delta u, \delta u_y),(\widetilde{\delta} u,\widetilde{\delta} u_y)) = \int_0^{\mathbf{p}} \left(\delta u(x)\widetilde{\delta} u_y(x) - \widetilde{\delta} u(x)\delta u_y(x) \right)\,dx
$$
and the Poisson bracket reads
$$
\{f,g\} = \int_0^{\mathbf{p}} \langle\nabla f, J \nabla g\rangle \,dx \; \text{ with } \; J=\begin{pmatrix} 0 & 1 \\ -1 & 0 \end{pmatrix}.
$$
Here $f$ and $g$ are functionals of the form $h: M^{\mathbf{p}} \to \CC,\; (u,u_y)\mapsto h(u,u_y)$ and $\nabla h$ denotes the corresponding gradient of $h$ in the function space $W^{1,2}(\RR/\mathbf{p}\ZZ) \times L^2(\RR/\mathbf{p}\ZZ)$. Note, that $\{f,g\} = \Omega(\nabla f, \nabla g)$ and
$$
\frac{d}{dt}h((u,u_y) + t(\delta u,\delta u_y))\bigg|_{t=0} = \Omega(\nabla h(u,u_y), (\delta u, \delta u_y))
$$
holds.

\subsection{A variational formula}

Considering $M_{\lambda}: M^{\mathbf{p}} \to SL(2,\CC)$ as a map on the symplectic manifold $M^{\mathbf{p}}$, we get the following lemma.

\begin{lemma}\label{lemma_variation_ln_mu}
For the map $(u,u_y) \mapsto M_{\lambda}$ we have the variational formula
$$
\delta M_{\lambda} = \left( \begin{smallmatrix} \delta a & \delta b \\ \delta c & \delta d \end{smallmatrix}\right) = \left(\int_0^{\mathbf{p}} F_{\lambda}(x) \delta U_{\lambda}(x) F_{\lambda}^{-1}(x) \,dx\right) M_{\lambda}
$$
with
$$
\delta U_{\lambda} = \frac{1}{2}
\begin{pmatrix}
-i\delta u_y & i\lambda^{-1} e^u \delta u -i e^{-u}\delta u \\
i\lambda  e^u\delta u -i e^{-u}\delta u & i\delta u_y
\end{pmatrix}.
$$
\end{lemma}

\begin{proof}
We follow the ansatz presented in \cite{McKean}, Section $6$, and obtain for $F_{\lambda}(x)$ solving $\frac{d}{dx}F_{\lambda} = F_{\lambda}U_{\lambda}$ with $F_{\lambda}(0) = \unity$ the variational equation
$$
\frac{d}{dx}\frac{d}{dt}F_{\lambda}(\delta u,\delta u_y)\bigg|_{t=0} = \left(\frac{d}{dt}F_{\lambda}(\delta u,\delta u_y)\bigg|_{t=0}\right) U_{\lambda} + F_{\lambda}\delta U_{\lambda} 
$$
with 
$$
\left(\frac{d}{dt}F_{\lambda}(\delta u,\delta u_y)\bigg|_{t=0}\right)(0) = \begin{pmatrix} 0 & 0 \\ 0 & 0 \end{pmatrix}
$$
and
$$
\delta U_{\lambda} = \frac{1}{2}
\begin{pmatrix}
-i\delta u_y & i\lambda^{-1} e^u \delta u -i e^{-u}\delta u \\
i\lambda  e^u\delta u -i e^{-u}\delta u & i\delta u_y
\end{pmatrix}.
$$
The solution of this differential equation is given by
$$
\left(\frac{d}{dt}F_{\lambda}(\delta u,\delta u_y)\bigg|_{t=0}\right)(x) = \left(\int_0^x F_{\lambda}(y) \delta U_{\lambda}(y) F_{\lambda}^{-1}(y) \,dy\right) F_{\lambda}(x)
$$
and evaluating at $x = \mathbf{p}$ yields
$$
\frac{d}{dt}M_{\lambda}(\delta u,\delta u_y)|_{t=0} = \left(\int_0^{\mathbf{p}} F_{\lambda}(y) \delta U_{\lambda}(y) F_{\lambda}^{-1}(y) \,dx\right) M_{\lambda}.
$$
This proves the claim.
\end{proof}


\section{Darboux coordinates}

Let us prove the existence of certain Darboux coordinates. First, recall the following theorem of Darboux.

\begin{theorem}[Darboux]
Locally a symplectic manifold $(M,\Omega)$ of dimension $2n$ is symplectomorphic to an open subset of $(\RR^{2n},\Omega_0)$, where the symplectic form $\Omega_0$ is given by
$$
\Omega_0 = \sum_{i=1}^n d\theta_i \wedge dI_i.
$$
\end{theorem}
That is, given a point $p \in M$, there is a neighborhood $V$ of $p$ in $M$ and a diffeomorphism $\Phi: U \subset \RR^{2n} \to V$ of an open subset $U$ in $\RR^{2n}$ onto $V$ such that
$$
\Phi^*\Omega = \Omega_0
$$
holds. The coordinates provided by $\Phi$ are called \textit{Darboux coordinates}. Let us recall a result in \cite{Schmidt_infinite}, where the non-linear Schr\"odinger operator with periodic potential $q(x)$ was investigated. It was shown that the points $(\lambda_i,\mu_i)_{i\in\mathcal{I}}$ of the corresponding divisor $D(q)$ are almost Darboux coordinates in the sense that
$$
\Omega(\delta q, \delta \widetilde{q}) = \sum_i (\tfrac{d}{dt}\lambda_i(\delta q)|_{t=0}) (\tfrac{d}{dt}\ln\mu_i(\delta \widetilde{q})|_{t=0}) - (\tfrac{d}{dt}\lambda_i(\delta \widetilde{q})|_{t=0}) (\tfrac{d}{dt}\ln\mu_i(\delta q)|_{t=0}),
$$
or in short form
$$
\Omega = \sum_i d\lambda_i \wedge d\ln\mu_i.
$$

An analogous result was proven in \cite{Adams_Harnad_Hurtubise} for finite dimensional integrable systems. In case of the KdV equation, Theorem 2.8 in \cite{Poeschel_Trubowitz} shows that the gradients of the points from the Dirichlet divisor are a symplectic basis on the tangent space by explicitly evaluating the symplectic form on these gradients. This shows that the points from the Dirichlet divisor are indeed Darboux coordinates. We will now adapt this idea to the present setting.

\subsection{Basic identities}
Let $\varphi$ and $\psi$ be solutions of the differential equations
\begin{footnotesize}
\beq
\frac{d}{dx} \begin{pmatrix} \varphi_1 \\ \varphi_2 \end{pmatrix} = 
-\frac{1}{2} \begin{pmatrix} -i u_y & i\lambda^{-1} e^u + i e^{-u}\\ i\lambda e^u +i e^{-u} & i u_y \end{pmatrix} \begin{pmatrix} \varphi_1 \\ \varphi_2 \end{pmatrix}, \;\;
\frac{d}{dy} \begin{pmatrix} \varphi_1 \\ \varphi_2 \end{pmatrix} = 
-\frac{1}{2} \begin{pmatrix} i u_x & -\lambda^{-1}e^u + e^{-u} \\ \lambda e^u - e^{-u} & -i u_x \end{pmatrix} \begin{pmatrix} \varphi_1 \\ \varphi_2 \end{pmatrix}
\label{eq_ode_phi}
\ee
\end{footnotesize}
and
\begin{footnotesize}
\beq
\frac{d}{dx} \begin{pmatrix} \psi_1 \\ \psi_2 \end{pmatrix} = 
\frac{1}{2} \begin{pmatrix} -i u_y & i\lambda e^u +i e^{-u} \\ i\lambda^{-1} e^u + i e^{-u} & i u_y \end{pmatrix} \begin{pmatrix} \psi_1 \\ \psi_2 \end{pmatrix}, \;\;
\frac{d}{dy} \begin{pmatrix} \psi_1 \\ \psi_2 \end{pmatrix} = 
\frac{1}{2} \begin{pmatrix} i u_x & \lambda e^u - e^{-u} \\ -\lambda^{-1} e^u + e^{-u} & -i u_x \end{pmatrix} \begin{pmatrix} \psi_1 \\ \psi_2 \end{pmatrix}.
\label{eq_ode_psi}
\ee
\end{footnotesize}
For now, we omit the initial values for $\varphi$ and $\psi$; they will play an important role later on. Considering the above differential equations, straightforward calculations show the following identities.
\begin{lemma}\label{basic_eqs}
Let $\omega := \psi_1\varphi_1 - \psi_2\varphi_2$. The following equations hold:
\begin{enumerate}
\item $\frac{d}{dx}(\varphi_1 \varphi_2) = -\frac{i}{2} ((\lambda e^u + e^{-u})\varphi_1^2 + (\lambda^{-1}e^u + e^{-u})\varphi_2^2)$
\item $\frac{d}{dy}(\varphi_1 \varphi_2) = \frac{1}{2} ((-\lambda e^u + e^{-u})\varphi_1^2 + (\lambda^{-1}e^u - e^{-u})\varphi_2^2)$
\item $\frac{d}{dx} \omega = (i\lambda e^u + i e^{-u})\varphi_1\psi_2 - (i\lambda^{-1}e^u + i e^{-u}) \psi_1\varphi_2$
\item $\frac{d}{dy} \omega = (\lambda e^u - e^{-u})\varphi_1\psi_2 + (\lambda^{-1}e^u - e^{-u}) \psi_1\varphi_2$
\item $\frac{d}{dx}\varphi_1^2 = i(u_y \varphi_1^2 - (\lambda^{-1}e^u + e^{-u})\varphi_1\varphi_2)$
\item $\frac{d}{dx}\varphi_2^2 = i(-u_y \varphi_2^2 - (\lambda e^u + e^{-u})\varphi_1\varphi_2)$
\item $\frac{d}{dx}(\psi_1\varphi_1) = i(\lambda e^u + e^{-u})\psi_2\varphi_1$
\item $\frac{d}{dx}(\psi_2\varphi_2) = i(\lambda^{-1} e^u + e^{-u})\psi_1\varphi_2$
\item $\frac{d}{dx}(\psi_1\varphi_2) = i(-u_y \psi_1\varphi_2 - \frac{1}{2}(\lambda e^u + e^{-u})\omega$
\item $\frac{d}{dx}(\psi_2\varphi_1) = i(u_y \psi_2\varphi_1 + \frac{1}{2}(\lambda^{-1} e^u + e^{-u})\omega$
\end{enumerate}
\end{lemma}

Let us now evaluate the above expressions at different values of $\lambda$. Utilizing the formulas from Lemma \ref{basic_eqs}, we get the following lemma.

\begin{lemma}\label{antiderivatives}
The following differential equations hold for $\lambda_i \neq \lambda_j$:
\begin{eqnarray*}
& \frac{d}{dx} & \left[\tfrac{-2i}{\lambda_i - \lambda_j}\left(\lambda_i \varphi_1^2(\lambda_i)(\varphi_2^2)(\lambda_j) - (\lambda_i + \lambda_j)(\varphi_1\varphi_2)(\lambda_i)(\varphi_1\varphi_2)(\lambda_j) + \lambda_j \varphi_2^2(\lambda_i)(\varphi_1^2)(\lambda_j) \right)\right] \\
& & = (\varphi_1\varphi_2)(\lambda_i)\del_y (\varphi_1\varphi_2)(\lambda_j) - (\varphi_1\varphi_2)(\lambda_i)\del_y (\varphi_1\varphi_2)(\lambda_j), \\
& \frac{d}{dx} & \left[\tfrac{-i}{\lambda_i - \lambda_j}\left(2\lambda_i \varphi_1^2(\lambda_i)(\psi_1\varphi_2)(\lambda_j) - (\lambda_i + \lambda_j)(\varphi_1\varphi_2)(\lambda_i)\omega(\lambda_j) - 2\lambda_j \varphi_2^2(\lambda_i)(\psi_2\varphi_1)(\lambda_j) \right)\right] \\
& & = (\varphi_1\varphi_2)(\lambda_i)\del_y \omega(\lambda_j) - \omega(\lambda_i)\del_y (\varphi_1\varphi_2)(\lambda_j), \\
& \frac{d}{dx} & \left[\tfrac{i}{\lambda_i - \lambda_j}\left( 4\lambda_j (\psi_1\varphi_2)(\lambda_i)(\psi_2\varphi_1)(\lambda_j) + (\lambda_i + \lambda_j)\omega(\lambda_i)\omega(\lambda_j) + 4\lambda_i(\psi_2\varphi_1)(\lambda_i)(\psi_1\varphi_2)(\lambda_j)\right)\right] \\
& & = \omega(\lambda_i)\del_y \omega(\lambda_j) - \omega(\lambda_j)\del_y \omega(\lambda_i).
\end{eqnarray*}
Moreover, there holds
\begin{gather*}
\frac{d}{dx} \left[-i(\varphi_1^2\varphi_2\psi_1)(\lambda_i) -i(\varphi_1\varphi_2^2\psi_2)(\lambda_i)\right] \\ 
= -\frac{1}{2}\left((\varphi_1^2\varphi_2\psi_2 + \varphi_1^3\psi_1)(\lambda_i)(\lambda_i e^u - e^{-u}) + (\varphi_1\varphi_2^2\psi_1 + \varphi_2^3\psi_2)(\lambda_i)(\lambda_i^{-1}e^u - e^{-u}) \right).
\end{gather*}
\end{lemma}

\subsection{Divisor points}
\label{subsection_divisor_dp}
By definition the monodromy at a divisor point $(\lambda_i,\mu_i)$ is given by
$$
M_{\lambda_i} = \begin{pmatrix} 1/\mu_i & 0 \\ c(\lambda_i) & \mu_i \end{pmatrix}.
$$
We will distinguish between points $(\lambda_i,\mu_i)$ such that $M_{\lambda_i} - \mu_i\unity$ has either a one- or a two-dimensional kernel. These two cases correspond to one- and two-dimensional eigenspaces, respectively. If the kernel at $(\lambda_i,\mu_i)$ is two-dimensional, one has $\mu_i = \pm 1 = \frac{\Delta(\lambda_i)}{2}$ and
$$
M_{\lambda_i} = \begin{pmatrix} \pm 1 & 0 \\ 0 & \pm 1 \end{pmatrix}.
$$
Here $\Delta(\lambda) := \trace(M_{\lambda})$. Consequently, the expression
$$
\det\left(M_{\lambda} - \tfrac{\Delta(\lambda)}{2}\unity\right) = 1-\Delta(\lambda)\tfrac{\Delta(\lambda)}{2}+\left(\tfrac{\Delta(\lambda)}{2}\right)^2 = -\tfrac{1}{4}(\Delta^2(\lambda) - 4)
$$
has at least a double root at $\lambda = \lambda_i$, i.e. $\Delta'(\lambda_i) = 0$. In particular, these points correspond to singularities on the spectral curve $Y$. \\

We can now distinguish the following three cases for eigenvectors $v(\lambda_i,\mu_i)$ of $M_{\lambda_i}$ and eigenvectors $w(\lambda_i,\mu_i)$ of the transposed monodromy $M_{\lambda_i}^t$ at a divisor point $(\lambda_i,\mu_i)$: 
\begin{enumerate}
\item[(i)] If the kernel at $(\lambda_i,\mu_i)$ is one-dimensional:
\begin{enumerate}
\item[(a)] For $\mu_i \neq \pm 1$: Setting $x_i = \frac{c(\lambda_i)}{\mu_i - 1/\mu_i}$ we get
$$
v(\lambda_i,\mu_i) = \begin{pmatrix} 0 \\ 1 \end{pmatrix}, \; v(\lambda_i,1/\mu_i) = \begin{pmatrix} 1 \\ -x_i \end{pmatrix}, \; w(\lambda_i,\mu_i) = \begin{pmatrix} x_i \\ 1 \end{pmatrix}, \; w(\lambda_i,1/\mu_i) = \begin{pmatrix} 1 \\ 0 \end{pmatrix}.
$$
\item[(b)] For $\mu_i = \pm 1$ we have $c(\lambda_i) \neq 0$. A direct calculation gives
$$
v(\lambda_i,\mu_i) = \begin{pmatrix} 0 \\ 1 \end{pmatrix}, \; v_{gen}(\lambda_i,\mu_i) = \begin{pmatrix} 1/c(\lambda_i) \\ 0 \end{pmatrix}, \; w(\lambda_i,\mu_i) = \begin{pmatrix} 1 \\ 0 \end{pmatrix}, \; w_{gen}(\lambda_i,\mu_i) = \begin{pmatrix} 0 \\ 1/c(\lambda_i) \end{pmatrix}.
$$
Here $M_{\lambda_i}$ and $M_{\lambda_i}^t$ take Jordan normal form with respect to the basis $v,\, v_{gen}$ and $w, \, w_{gen}$, respectively.
\end{enumerate}

\item[(ii)]If the kernel at $(\lambda_i,\mu_i)$ is two-dimensional, every vector is an eigenvector of $M_{\lambda_i} = \pm \unity$ and we set
$$
v(\lambda_i,\mu_i) = \begin{pmatrix} 0 \\ 1 \end{pmatrix}, \; v(\lambda_i,1/\mu_i) = \begin{pmatrix} 1 \\ 0 \end{pmatrix}, \; w(\lambda_i,\mu_i) = \begin{pmatrix} 0 \\ 1 \end{pmatrix}, \; w(\lambda_i,1/\mu_i) = \begin{pmatrix} 1 \\ 0 \end{pmatrix}.
$$
\end{enumerate}

Setting $\varphi(0) = \psi(0) = (0,1)^t$, we get the following identities.

\begin{lemma}\label{lemma_identities}
For $\varphi = F_{\lambda}^{-1}(0,1)^t$ and $\psi = F_{\lambda}^t (0,1)^t$ we have the following identities at all divisor points $(\lambda_i,\mu_i)$:
\begin{enumerate}
\item $\varphi(\mathbf{p}) = \mu_i^{-1}\varphi(0)$ and $\omega(\mathbf{p}) = \omega(0) = -1$.
\item $(1,0) F_{\lambda} = - \varphi^t J$.
\end{enumerate}
\end{lemma}

\begin{proof} The following observations are based on subsection \ref{subsection_divisor_dp}.
\begin{enumerate}
\item In all 3 cases (i)(a), (i)(b) and (ii), a direct calculation leads to
$$
\varphi(\mathbf{p}) = F_{\lambda_i}^{-1}(\mathbf{p})(0,1)^t = M_{\lambda_i}^{-1} (0,1)^t = (0,\mu_i^{-1})^t = \mu_i^{-1}\varphi(0).
$$
and
$$
\psi(\mathbf{p}) = F_{\lambda_i}^t(\mathbf{p})(0,1)^t = M_{\lambda_i}^t (0,1)^t = (c(\lambda_i), \mu_i)^t.
$$
Moreover, we get
$$
\omega(\mathbf{p}) = (\psi_1\varphi_1 - \psi_2\varphi_2)(\mathbf{p}) = -1 = \omega(0).
$$
\item Due to $\det(F_{\lambda}) = 1$ we get $F_{\lambda}^{-1} = -J F_{\lambda}^t J$ and consequently $(1,0) F_{\lambda} = - \varphi^t J$ holds. 
\end{enumerate}
\end{proof}

\subsection{A symplectic basis for $T_{(u,u_y)}M^{\mathbf{p}}$}
                                           
Let us adapt Theorem 8 from Chapter 2 in \cite{Poeschel_Trubowitz}.

\begin{theorem}\label{theorem_PT}
Denote by $\{(\lambda_i,\mu_i)\}$ the set of all divisor points on $Y$ and let Assumption \ref{assumption_divisor} hold. Moreover, let $\varphi$ be a solution of \eqref{eq_ode_phi} and $\psi$ a solution of \eqref{eq_ode_psi} and set $\varphi(0) = \psi(0):= (0,1)^t$. Finally, let $a_i := (\varphi_1\varphi_2(\lambda_i), \del_y (\varphi_1\varphi_2)(\lambda_i))$, $b_i := (\omega(\lambda_i),\del_y\omega(\lambda_i))$ and
$$
\kappa_i := \int_0^{\mathbf{p}} \left( \lambda_i \varphi_1^2(\lambda_i) e^u + \lambda_i^{-1} \varphi_2^2(\lambda_i) e^u \right) \, dx.
$$
Then there holds for all $i,j$:
\begin{enumerate}
\item $\Omega(a_i, a_j) = 0$,
\item $\Omega(a_i, b_j) = \delta_{ij} \cdot \kappa_i$,
\item $\Omega(b_i, b_j) = 0$.
\end{enumerate}
\end{theorem}

\begin{remark}
The proof of Theorem \ref{Theorem_Darboux} will show that $\partial_\lambda b(\lambda_i)=-i(\mu_i/2\lambda_i)\kappa_i.$ Due to Assumption \ref{assumption_divisor}, the derivative $\partial_\lambda b(\lambda_i) \neq 0$ and therefore $\kappa_i \neq 0$.
\end{remark}

\begin{proof}[Proof of Theorem \ref{theorem_PT}]
The following calculations are based on Lemma \ref{antiderivatives} and Lemma \ref{lemma_identities}.
\begin{enumerate}

\item The case $i=j$ is trivial. For $i \neq j$ we get
\begin{eqnarray*}
\Omega(a_i, a_j) & = & \int_0^{\mathbf{p}} \left( (\varphi_1\varphi_2)(\lambda_i)\del_y (\varphi_1\varphi_2)(\lambda_j) - (\varphi_1\varphi_2)(\lambda_j)\del_y (\varphi_1\varphi_2)(\lambda_i)  \right) \, dx \\
& \stackrel{\text{Lemma }\ref{antiderivatives}}{=} & \big[\tfrac{-2i}{\lambda_i - \lambda_j}(\lambda_i \varphi_1^2(\lambda_i)(\varphi_2^2)(\lambda_j) - (\lambda_i + \lambda_j)(\varphi_1\varphi_2)(\lambda_i)(\varphi_1\varphi_2)(\lambda_j) \\
& & + \lambda_j \varphi_2^2(\lambda_i)(\varphi_1^2)(\lambda_j))\big]_0^{\mathbf{p}}.
\end{eqnarray*}
Lemma \ref{lemma_identities} gives $\varphi_1(\lambda_k)(0) = \varphi_1(\lambda_k)(\mathbf{p}) = 0$ for $k=i,j$. Thus, we have $\Omega(a_i, a_j) = 0$.

\item For $i \neq j$ we get
\begin{eqnarray*}
\Omega(a_i, b_j) & = & \int_0^{\mathbf{p}} (\varphi_1\varphi_2)(\lambda_i)\del_y \omega(\lambda_j) - \omega(\lambda_i)\del_y (\varphi_1\varphi_2)(\lambda_j)\, dx \\
& \stackrel{\text{Lemma }\ref{antiderivatives}}{=} & \big[\tfrac{-i}{\lambda_i - \lambda_j}(2\lambda_i \varphi_1^2(\lambda_i)(\psi_1\varphi_2)(\lambda_j) - (\lambda_i + \lambda_j)(\varphi_1\varphi_2)(\lambda_i)\omega(\lambda_j) \\
& & - 2\lambda_j \varphi_2^2(\lambda_i)(\psi_2\varphi_1)(\lambda_j) )\big]_0^{\mathbf{p}}.
\end{eqnarray*}
Lemma \ref{lemma_identities} gives $\varphi_1(\lambda_k)(0) = \varphi_1(\lambda_k)(\mathbf{p}) = 0$ for $k=i,j$. Thus, we have $\Omega(a_i, b_j) = 0$. In the case $i = j$ we claim $\Omega(a_i, b_i) = \kappa_i$. Since $\psi_1\varphi_1 + \psi_2\varphi_2 \equiv 1$, we can calculate
\begin{eqnarray*}
\Omega(a_i, b_i) - \kappa_i & = & \int_0^{\mathbf{p}} (\varphi_1\varphi_2)(\lambda_i)\del_y \omega(\lambda_i) - \omega(\lambda_i)\del_y (\varphi_1\varphi_2)(\lambda_i)\, dx \\
& & -\int_0^{\mathbf{p}} \left( \lambda_i \varphi_1^2(\lambda_i) e^u + \lambda_i^{-1} \varphi_2^2(\lambda_i) e^u \right) \, dx \\
& = & \int_0^{\mathbf{p}} (\varphi_1\varphi_2)(\lambda_i)\del_y \omega(\lambda_i) - \omega(\lambda_i)\del_y (\varphi_1\varphi_2)(\lambda_i)\, dx \\
& & -\int_0^{\mathbf{p}} (\psi_1\varphi_1 + \psi_2\varphi_2)\left( \lambda_i \varphi_1^2(\lambda_i) e^u + \lambda_i^{-1} \varphi_2^2(\lambda_i) e^u \right) \, dx \\
& = & -\frac{1}{2}\int_0^{\mathbf{p}} (\varphi_1^2\varphi_2\psi_2 + \varphi_1^3\psi_1)(\lambda_i)(\lambda_i e^u - e^{-u}) \, dx \\
& & -\frac{1}{2}\int_0^{\mathbf{p}} (\varphi_1\varphi_2^2\psi_1 + \varphi_2^3\psi_2)(\lambda_i)(\lambda_i^{-1}e^u - e^{-u}) \, dx \\
& \stackrel{\text{Lemma }\ref{antiderivatives}}{=} & \left[-i(\varphi_1^2\varphi_2\psi_1)(\lambda_i) -i(\varphi_1\varphi_2^2\psi_2)(\lambda_i)\right]_0^{\mathbf{p}}.
\end{eqnarray*}
Again, Lemma \ref{lemma_identities} gives $\varphi_1(\lambda_i)(0) = \varphi_1(\lambda_i)(\mathbf{p}) = 0$ and consequently $\Omega(a_i, b_i) - \kappa_i = 0$.

\item The case $i=j$ is trivial. For $i \neq j$ we get
\begin{eqnarray*}
\Omega(b_i, b_j) & = & \int_0^{\mathbf{p}} \left( \omega(\lambda_i)\del_y \omega(\lambda_j) - \omega(\lambda_j)\del_y \omega(\lambda_i)  \right) \, dx \\
& \stackrel{\text{Lemma }\ref{antiderivatives}}{=} & \big[\tfrac{i}{\lambda_i - \lambda_j}(4\lambda_j (\psi_1\varphi_2)(\lambda_i)(\psi_2\varphi_1)(\lambda_j) + (\lambda_i + \lambda_j)\omega(\lambda_i)\omega(\lambda_j) \\
& & + 4\lambda_i(\psi_2\varphi_1)(\lambda_i)(\psi_1\varphi_2)(\lambda_j))\big]_0^{\mathbf{p}}.
\end{eqnarray*}
Lemma \ref{lemma_identities} gives $\varphi_1(\lambda_k)(0) = \varphi_1(\lambda_k)(\mathbf{p}) = 0$ and $\omega(\lambda_k)(0) = \omega(\lambda_k)(\mathbf{p}) = -1$ for $k=i,j$. Thus, we have $\Omega(b_i, b_j) = 0$. This concludes the proof.
\end{enumerate}
\end{proof}

\subsection{$(\lambda_i,\mu_i)$ considered as Darboux coordinates}
In the following a variation $\frac{d}{dt} f(\delta u, \delta u_y)|_{t=0}$ of a function $f: M^{\mathbf{p}} \to \CC$ in the direction $(\delta u, \delta u_y)$ is abbreviated by $\delta f$. Likewise, a variation of $f$ in the direction $(\widetilde{\delta} u, \widetilde{\delta} u_y)$ is denoted by $\widetilde{\delta}f$.

\begin{theorem}\label{Theorem_Darboux}
Let Assumption \ref{assumption_divisor} hold for a divisor $D(u,u_y)$ on $Y$. For all $(\delta u, \delta u_y), (\widetilde{\delta} u, \widetilde{\delta} u_y) \in T_{(u,u_y)}M^{\mathbf{p}}$ that are finite linear combinations of the $a_i, b_i$ from Theorem \ref{theorem_PT} there holds
\beq
\Omega((\delta u, \delta u_y),(\widetilde{\delta} u, \widetilde{\delta}u_y)) = \frac{i}{2}\sum_i \left(\frac{\delta \lambda_i}{\lambda_i}\frac{\widetilde{\delta}\mu_i}{\mu_i} - \frac{\widetilde{\delta}\lambda_i}{\lambda_i}\frac{\delta\mu_i}{\mu_i}\right)
\label{eq_Darboux}
\ee
\end{theorem}

\begin{remark}
In order to prove equation \eqref{eq_Darboux} in full generality, one would have to establish the following claims:
\begin{enumerate}
\item[(i)] Every element in $T_{(u,u_y)}M^{\mathbf{p}}$ is a linear combination of the $a_i,\, b_i$.
\item[(ii)] The sum on the right hand side of equation \eqref{eq_Darboux} converges.
\end{enumerate}
This can be achieved by applying the careful asymptotic analysis carried out in \cite{Klein}, section 14.
\end{remark}

\begin{proof}[Proof of Theorem \ref{Theorem_Darboux}]
Due to Lemma \ref{lemma_variation_ln_mu} we obtain 
$$
\delta M_{\lambda} = \left( \begin{smallmatrix} \delta a & \delta b \\ \delta c & \delta d \end{smallmatrix}\right) = \left(\int_0^{\mathbf{p}} F_{\lambda}(x) \delta U_{\lambda}(x) F_{\lambda}^{-1}(x) \,dx\right) M_{\lambda}
$$
and thus
$$
\delta b(\lambda_i) = \begin{pmatrix} 1 & 0 \end{pmatrix} \delta M_{\lambda_i} \begin{pmatrix} 0 \\ 1 \end{pmatrix}.
$$

A direct calculation gives
\begin{eqnarray*}
\delta b(\lambda_i) & = & \mu_i\int_0^{\mathbf{p}} (1,0)F_{\lambda_i}(x) \delta U_{\lambda_i}(x) \varphi(\lambda_i) \,dx \\
& \stackrel{\text{Lemma }\ref{lemma_identities}}{=} & -\mu_i\int_0^{\mathbf{p}} \varphi^t(\lambda_i) J \delta U_{\lambda_i}(x) \varphi(\lambda_i) \,dx \\
& = & -i\mu_i\int_0^{\mathbf{p}}\left(\varphi_1\varphi_2(\lambda_i) \delta u_y - \tfrac{1}{2} ((-\lambda_i e^u + e^{-u})\varphi_1^2(\lambda_i) + (\lambda_i^{-1}e^u - e^{-u})\varphi_2^2(\lambda_i)) \delta u \right) dx \\
& = & -i\mu_i\int_0^{\mathbf{p}}\left(\varphi_1\varphi_2(\lambda_i) \delta u_y - \del_y(\varphi_1\varphi_2)(\lambda_i) \delta u \right) dx \\
& = & -i\mu_i\, \Omega(a_i,(\delta u,\delta u_y)).
\end{eqnarray*}

Considering the entry
$$
\delta d(\lambda_i) = \begin{pmatrix} 0 & 1 \end{pmatrix} \delta M_{\lambda_i} \begin{pmatrix} 0 \\ 1 \end{pmatrix}
$$
we get
\begin{eqnarray*}
\delta d(\lambda_i) & = & \mu_i \int_0^{\mathbf{p}} \psi^t (\lambda_i) \delta U_{\lambda_i}(x) \varphi(\lambda_i) \,dx \\
& = & -i\mu_i\, \Omega((\omega(\lambda_i),\del_y\omega(\lambda_i)),(\delta u,\delta u_y)) = -i\mu_i\, \Omega(b_i,(\delta u,\delta u_y)).
\end{eqnarray*}

Finally, in analogy to the variation $\delta b(\lambda_i)$, the expression 
$$
\del_{\lambda}b(\lambda_i) = \begin{pmatrix} 1 & 0 \end{pmatrix} \del_{\lambda} M_{\lambda_i} \begin{pmatrix} 0 \\ 1 \end{pmatrix}
$$
is given by
\begin{eqnarray*}
\del_{\lambda}b(\lambda_i) & = & \mu_i\int_0^{\mathbf{p}} (0,1) F_{\lambda_i}(x) \del_{\lambda} U_{\lambda_i}(x) \varphi(\lambda_i) \,dx \\
& \stackrel{\text{Lemma }\ref{lemma_identities}}{=} & -\mu_i\int_0^{\mathbf{p}} \varphi^t(\lambda_i) J \del_{\lambda} U_{\lambda_i}(x) \varphi(\lambda_i) \,dx \\
& = & -i\frac{\mu_i}{2\lambda_i} \int_0^{\mathbf{p}} \left( \lambda_i \varphi_1^2(\lambda_i) e^u + 1/\lambda_i \varphi_2^2(\lambda_i) e^u \right) \, dx = -i\frac{\mu_i}{2\lambda_i} \cdot \kappa_i.
\end{eqnarray*}

Due to Theorem \ref{theorem_PT} there holds
$$
(\delta u, \delta u_y) = \sum_i \frac{1}{\kappa_i} \bigg[\Omega(a_i,(\delta u, \delta u_y))\, b_i - \Omega(b_i,(\delta u, \delta u_y))\, a_i \bigg]
$$
and consequently
\begin{small}
$$
\Omega((\delta u, \delta u_y),(\widetilde{\delta} u, \widetilde{\delta}u_y)) = \sum_{i} \frac{1}{\kappa_i}\left[\Omega(a_i,(\widetilde{\delta} u,\widetilde{\delta} u_y)) \Omega(b_i,(\delta u,\delta u_y)) - \Omega(a_i,(\delta u,\delta u_y)) \Omega(b_i,(\widetilde{\delta} u,\widetilde{\delta} u_y))\right].
$$
\end{small}

On the other hand, the definition of $D$ implies $b(\lambda_i) = 0,\, d(\lambda_i) = \mu_i$, and we get
\begin{eqnarray*}
\delta \lambda_i & = & -\tfrac{\delta b(\lambda_i)}{\del_{\lambda} b(\lambda_i)}, \\
\delta \mu_i & = & \delta d(\lambda_i) - \frac{\del_{\lambda}d(\lambda_i) \delta b(\lambda_i)}{\del_{\lambda}b(\lambda_i)}
\end{eqnarray*}
by the Implicit Function Theorem. Using the last equations we can calculate
\begin{eqnarray*}
\frac{i}{2}\sum_{i} \left(\frac{\delta \lambda_i}{\lambda_i}\frac{\widetilde{\delta}\mu_i}{\mu_i} - \frac{\widetilde{\delta}\lambda_i}{\lambda_i}\frac{\delta\mu_i}{\mu_i}\right)
& = & \frac{i}{2}\sum_{i} \frac{1}{\lambda_i \mu_i}\left(\delta\lambda_i \widetilde{\delta}\mu_i - \widetilde{\delta}\lambda_i \delta\mu_i\right) \\
& = & \frac{i}{2}\sum_{i} \frac{1}{\lambda_i \mu_i \del_{\lambda}b(\lambda_i)}\left(\widetilde{\delta}b(\lambda_i) \delta d(\lambda_i) - \delta b(\lambda_i) \widetilde{\delta} d(\lambda_i)\right).
\end{eqnarray*}

Inserting the expressions for $\delta b(\lambda_i),\, \delta d(\lambda_i)$ and $\del_{\lambda}b(\lambda_i)$ yields
$$
\Omega((\delta u, \delta u_y),(\widetilde{\delta} u, \widetilde{\delta}u_y)) = \frac{i}{2}\sum_i \left(\frac{\delta \lambda_i}{\lambda_i}\frac{\widetilde{\delta}\mu_i}{\mu_i} - \frac{\widetilde{\delta}\lambda_i}{\lambda_i}\frac{\delta\mu_i}{\mu_i}\right)
$$
and the claim is proved.
\end{proof}


\bibliography{diss_bibliography} 

\begin{thebibliography}{1}

\bibitem{Adams_Harnad_Hurtubise}
M.~{Adams}, J.~{Harnad}, and J.~{Hurtubise}.
\newblock {Darboux coordinates and Liouville-Arnold integration in loop
  algebras.}
\newblock {\em {Commun. Math. Phys.}}, 155(2):385--413, 1993.

\bibitem{Hitchin_Harmonic}
N.~J. Hitchin.
\newblock {Harmonic maps from a 2-torus to the 3-sphere.}
\newblock {\em J. Differ. Geom.}, 31(3):627--710, 1990.

\bibitem{Klein}
S.~Klein.
\newblock {\em A spectral theory for simply periodic solutions of the
  sinh-Gordon equation}.
\newblock Manuscript, submitted as habilitation thesis.

\bibitem{Knopf_phd}
M.~Knopf.
\newblock {\em Periodic solutions of the sinh-Gordon equation and integrable
  systems}.
\newblock PhD thesis, 2013.

\bibitem{McKean}
H.~P. McKean.
\newblock {The sine-Gordon and sinh-Gordon equations on the circle.}
\newblock {\em {Commun. Pure Appl. Math.}}, 34:197--257, 1981.

\bibitem{Poeschel_Trubowitz}
J.~P\"oschel and E.~Trubowitz.
\newblock {\em Inverse spectral theory.}
\newblock Pure and Applied Mathematics, Vol. 130. Boston: Academic Press Inc.,
  192 p., 1987.

\bibitem{Schmidt_infinite}
M.~U. Schmidt.
\newblock {\em {Integrable systems and Riemann surfaces of infinite genus.}}
\newblock Memoirs of the American Mathematical Society Series, Number 581,
  1996.

\end{thebibliography}
\bibliographystyle{abbrv}

\end{document}